\documentclass[11pt]{article}

\usepackage[a4paper,margin=1in]{geometry}
\usepackage[T1]{fontenc}
\usepackage{lmodern}
\usepackage{amsmath,amssymb,amsthm,mathtools}
\usepackage{microtype}
\usepackage[hidelinks]{hyperref}
\hypersetup{pdftitle={The Ding--Song--Sun inequality via a maximum principle},
 pdfauthor={Fu-Hsuan Ho},
 pdfsubject={A finite-graph proof for the Ellis--Monroe--Newman potential class}}

\newcommand{\R}{\mathbb R}
\newcommand{\dd}[1]{\,\mathrm d#1}
\DeclarePairedDelimiter{\abs}{\lvert}{\rvert}
\DeclareMathOperator{\Cov}{Cov}
\DeclareMathOperator{\Var}{Var}
\newcommand{\indic}{\mathbf 1}
\numberwithin{equation}{section}

\newtheorem{theorem}{Theorem}[section]
\newtheorem{lemma}[theorem]{Lemma}
\newtheorem{corollary}[theorem]{Corollary}
\theoremstyle{definition}
\newtheorem{definition}[theorem]{Definition}
\theoremstyle{remark}
\newtheorem{remark}[theorem]{Remark}

\title{The Ding--Song--Sun inequality via a maximum principle}
\author{Fu-Hsuan Ho\\
  \small Department of Mathematics, Weizmann Institute of Science\\
  \small 76100 Rehovot, Israel\\
  \small\texttt{fu-hsuan.ho@weizmann.ac.il}}
\date{}

\begin{document}
\maketitle
\vspace{-1.8em}

\begin{abstract}
We prove the Ding--Song--Sun inequality for continuous spin models on finite
ferromagnetic graphs whose even single-site potentials have convex
derivatives on the positive half of their domain. This class, introduced by Ellis,
Monroe and Newman, includes the $\varphi^4$ and sinh-Gordon potentials.
The proof uses a rank-one interpolation of the interactions and a maximum
principle for a cooperative parabolic system on a half-plane. The boundary
conditions follow by induction on the number of vertices and the number of
positive coordinates of the field increment. In particular, the truncated
two-point function in an arbitrary external field is bounded above by its
zero-field value.
\end{abstract}

\section{Introduction and main result}

Ding, Song and Sun \cite{DSS23} proved that, in a finite ferromagnetic
Ising model, the change in magnetization caused by a nonnegative field
increment is largest when the background field vanishes. Their inequality
implies that the truncated two-point function is maximal at zero external
field, even when the competing field has mixed signs. We prove the same
comparison for the potential class introduced by Ellis, Monroe and Newman
\cite{EMN76}.

Our argument combines the one-site GHS inequality with FKG positivity.
A rank-one perturbation of the interactions gives a heat equation for the
partition function and a cooperative parabolic equation for the
magnetization deficit. We prove positivity of this deficit by a maximum
principle on a two-variable half-plane. A choice of rank-one perturbation
that restores the edges incident to one vertex closes a double induction.
Symmetric truncation reduces the proof to bounded spins.

\paragraph{Acknowledgment}
Large Language Model is functioned at the level of coauthor throughout this 
project, which includes conducting numerical experiments, literature reviews and 
suggesting proof strategy. 

\paragraph{Related work.}
Abdelghani, Bauerschmidt, Bodineau and Dagallier
\cite[Theorem~1.1]{ABBD26} have recently proved the DSS inequality for
even $C^3$ potentials that are unbounded above and have convex derivatives on the
positive half-line, with a strictly positive definite ferromagnetic
quadratic form. Their proof uses a
Gaussian convolution of the potential and a parabolic maximum principle.
The rank-one interpolation used here gives a direct evolution of the
magnetization deficit on a half-plane. Both arguments use induction to
establish the parabolic boundary conditions.

\subsection{Setting and statement}

Let $G=(V,E)$ be a finite graph with $\abs{V}\geq1$. Write $J_{uv}=J_{vu}\geq0$
for its interaction strengths, set $J_{vv}=0$, and put $J_{uv}=0$ when
$\{u,v\}\notin E$. Thus all models can equivalently be viewed on the
complete graph, with some couplings equal to zero.
Fix an ordering of $V$; sums over $u<v$ count each
unordered pair once. Vector inequalities are understood coordinatewise,
and $e_v$ denotes the $v$th coordinate vector.
For single-site probability measures $(\nu_v)_{v\in V}$ and an external field
$g=(g_v)_{v\in V}\in\R^V$, define
\begin{equation}
\label{eq:gibbs}
 \mu_g(\dd{x})
 =\frac1{Z(g)}
 \exp\left\{\sum_{u<v}J_{uv}x_ux_v+\sum_{v\in V}g_vx_v\right\}
 \prod_{v\in V}\nu_v(\dd{x_v}).
\end{equation}
Let $\Cov_g$ denote covariance under $\mu_g$, and write
\[
 m_o(g)=\int x_o\,\mu_g(\dd{x})=\partial_{g_o}\log Z(g).
\]

\begin{definition}
\label{def:class}
At each vertex $v\in V$, let $0<a_v\leq\infty$ and let
$V_v\in C^1((-a_v,a_v))$ be even, with $V_v'$ convex on $[0,a_v)$.
Assume
\[
 Z_v:=\int_{-a_v}^{a_v}e^{-V_v(x)}\dd{x}<\infty,
\]
and define the probability measure
\begin{equation}
\label{eq:potential-class}
 \nu_v(\dd{x})
 =\frac1{Z_v}\indic_{\{\abs{x}<a_v\}}e^{-V_v(x)}\dd{x}.
\end{equation}
\end{definition}

\begin{remark}
These are the potential measures in \cite[Theorem~1.2(c),(d)]{EMN76}.
The theorem below concerns this potential class; it does not assert the
comparison for every measure in the abstract class $\mathcal G_-$ of
\cite{EMN76}.
Convexity of $V_v$ itself is not required. Examples include
\[
 V_v(x)=\alpha_v x^4+\beta_vx^2
 \quad\text{and}\quad
 V_v(x)=\alpha_v\cosh x+\beta_vx^2,
 \qquad \alpha_v>0,\quad \beta_v\in\R.
\]
These examples appear in \cite[(1.8)--(1.9)]{EMN76}; both give
$Z(f)<\infty$ for all couplings and fields.
\end{remark}

\begin{theorem}[Ding--Song--Sun inequality]
\label{thm:dss}
Suppose that all single-site measures satisfy Definition~\ref{def:class},
and that $Z(f)<\infty$ for every $f\in\R^V$. Then, for every $o\in V$,
$g\in\R^V$, and $h\in[0,\infty)^V$,
\begin{equation}
\label{eq:dss}
 m_o(g+h)-m_o(g-h)
 \leq m_o(h)-m_o(-h)=2m_o(h).
\end{equation}
Equivalently,
\begin{equation}
\label{eq:midpoint}
 2m_o(h)-m_o(h+g)-m_o(h-g)\geq0.
\end{equation}
\end{theorem}

The equivalence follows from global spin-flip symmetry,
$m_o(-f)=-m_o(f)$. The finiteness assumption ensures that $Z$ and every
$m_o$ are smooth on $\R^V$, and that all polynomial moments are finite.
Indeed, for fields in a fixed compact box, the absolute value of any
polynomial factor times the Gibbs weight is dominated by a constant times
the sum of the weights at the corners of a larger box. This justifies differentiation under
the integral; in particular,
\begin{equation}
\label{eq:field-derivative}
 \partial_{f_v}m_o(f)=\Cov_f(x_o,x_v).
\end{equation}

The Ding--Song--Sun inequality immediately gives the following zero-field
covariance bound.

\begin{corollary}
\label{cor:covariance}
Under the hypotheses of Theorem~\ref{thm:dss}, for every $g\in\R^V$ and
$o,v\in V$,
\begin{equation}
\label{eq:covariance-majorant}
 0\leq\Cov_g(x_o,x_v)\leq\Cov_0(x_o,x_v).
\end{equation}
\end{corollary}

\begin{proof}
Set $h=\eta e_v$ in \eqref{eq:dss}, divide by $2\eta$, and let
$\eta\downarrow0$, using \eqref{eq:field-derivative}.
The lower bound is the FKG inequality \eqref{eq:fkg}, extended from bounded
spins by the truncation argument below. The required second moments are
finite by the exponential-tilt bound following Theorem~\ref{thm:dss}.
\end{proof}

At zero field the means vanish by spin-flip symmetry, so the upper bound
in \eqref{eq:covariance-majorant} is also the untruncated two-point function
$\int x_ox_v\,\mu_0(\dd{x})$.

\section{One-site comparison and FKG positivity}

For a single-site law $\nu$ with potential $V$ as in
Definition~\ref{def:class}, write
\[
 m(s)=\frac{\mathrm d}{\mathrm ds}\log\int e^{sx}\nu(\dd{x}).
\]
These laws satisfy the ambient moment condition of \cite[p.~168]{EMN76}:
$\int e^{\kappa x^2}\nu(\dd{x})<\infty$ for some $\kappa>0$.
For bounded support this is immediate. On unbounded support, integrability
forces $V'(x_0)>0$ for some $x_0>0$. Convexity and $V'(0)=0$ then give
$V'(x)\geq xV'(x_0)/x_0$ for $x\geq x_0$. Integrating and using evenness
yields $V(x)\geq c_0x^2-C$ for some $c_0>0$ and $C<\infty$, so any
$0<\kappa<c_0$ works. In particular, $V(x)\to\infty$ as $\abs{x}\to\infty$,
as required in \cite[Theorem~1.2(c)]{EMN76}.
For bounded support, the density $p\propto e^{-V}$ is positive and $C^1$
on $(-a,a)$, and $p'/p=-V'$ is concave on $[0,a)$, exactly as required in
\cite[Theorem~1.2(d)]{EMN76}. Thus every law in Definition~\ref{def:class}
belongs to $\mathcal G_-$, and all its linear exponential moments are finite.

\begin{lemma}
\label{lem:one-site}
For every measure in Definition~\ref{def:class} and every $h\geq0$, the
map
\[
 g\longmapsto m(g+h)-m(g-h)
\]
is even and nonincreasing on $[0,\infty)$. In particular,
\[
 m(g+h)-m(g-h)\leq2m(h)\qquad(g\in\R).
\]

The single-site class is preserved by multiplication by $e^{-c x^2}$ and
renormalization, for every $c\geq0$. It is also preserved by symmetric
restriction to $[-L,L]$ and renormalization, for every $L>0$.
\end{lemma}

\begin{proof}
By symmetry, $m$ is odd and $m'$ is even.
Apply \cite[Theorem~1.1 and (1.4)]{EMN76} with $N=1$, $J_{11}=0$, and
$h_1=s\geq0$. The three differentiation indices may coincide, giving
$m''(s)\leq0$.
Thus the displayed map is even and, for $g\geq0$,
\[
 \frac{\mathrm d}{\mathrm dg}\bigl[m(g+h)-m(g-h)\bigr]
 =m'(g+h)-m'(\abs{g-h})\leq0.
\]

The quadratic modification replaces $V$ by $V+cx^2$.
Its derivative $V'(x)+2cx$ remains convex on the positive half-interval.
Symmetric restriction leaves this condition unchanged in the interior of
the new support; \cite[Theorem~1.2(d)]{EMN76} requires no density regularity
at the endpoints.
\end{proof}

We also use FKG in the following form:
\begin{equation}
\label{eq:fkg}
 \Cov_f(x_u,x_v)\geq0
 \qquad\text{for every }f\in\R^V.
\end{equation}
For bounded spins this follows from the FKG lattice condition for the
ferromagnetic density relative to the product reference measure
\cite{FKG71}. To obtain the statement for continuous reference measures,
approximate each reference measure by its pushforward under an increasing
finite-valued discretization. The resulting finite-spin Gibbs measures
satisfy FKG, and their first and second moments converge by bounded
convergence. The external field need not have a constant sign.
The truncation argument in Section~\ref{sec:truncation} extends
\eqref{eq:fkg} to the unbounded-spin models under consideration.
No multi-site GHS inequality is needed below.

\section{Rank-one interpolation and maximum principle}

\subsection{Reduction to bounded spins}
\label{sec:truncation}

We first reduce Theorem~\ref{thm:dss} to single-site measures supported in
$[-M,M]$ for a common finite $M>0$.
For $M>0$, let $\nu_v^{(M)}=\nu_v(\,\cdot\mid[-M,M])$.
These restrictions remain in the single-site class by Lemma~\ref{lem:one-site}.
The associated Gibbs measure at field $f$ is
$\mu_f(\,\cdot\mid[-M,M]^V)$.
The bound $\abs{x_o}\leq e^{x_o}+e^{-x_o}$ and the finiteness of
$Z(f+e_o)+Z(f-e_o)$ imply that $x_o$ is integrable under $\mu_f$.
Dominated convergence therefore gives, as $M\uparrow\infty$,
\[
 m_o^{(M)}(f)
 =\frac{\int_{[-M,M]^V}x_o\,\mu_f(\dd{x})}
        {\mu_f([-M,M]^V)}
 \longrightarrow m_o(f).
\]
Passing to the limit at $f=h,h+g,h-g$ shows that it suffices to prove
\eqref{eq:midpoint} for the restricted measures.
The same argument applies to $x_ox_v$, whose absolute integrability follows
from the moment bound above. Thus the covariances of the restricted models
also converge to those of the original model, which justifies the
unbounded-spin version of \eqref{eq:fkg}.

\subsection{Interpolation and positivity preservation}

For the rest of the proof, all single-site measures are supported in
$[-M,M]$. Let $K_{uv}=K_{vu}\geq0$ with $K_{vv}=0$,
$w\in[0,\infty)^V$, and
$0\leq t\leq1$. Define
\begin{equation}
\label{eq:interpolation}
 \begin{aligned}
 Z(t,f)=\int\exp\Bigg\{&\sum_{u<v}K_{uv}x_ux_v+f\cdot x
       +\frac t2(w\cdot x)^2
       -\frac12\sum_vw_v^2x_v^2\Bigg\}
       \prod_v\nu_v(\dd{x_v}),\\
 m_o(t,f)&=\partial_{f_o}\log Z(t,f),
 \qquad \partial_w=\sum_vw_v\partial_{f_v}.
 \end{aligned}
\end{equation}
Write $m(t,f)=(m_v(t,f))_{v\in V}$, and let $\Cov_{t,f}$ denote covariance
under the Gibbs measure obtained by normalizing the integrand in
\eqref{eq:interpolation}.
The off-diagonal couplings at time $t$ are $K_{uv}+tw_uw_v\geq0$.
The single-site laws are proportional to
$e^{-(1-t)w_v^2x_v^2/2}\nu_v(\dd{x_v})$.

Fix $k\in[0,\infty)^V$ and $g\in\R^V$. For $q\geq0$ and $r\in\R$, write
\[
 m^0=m(t,k+qw),\qquad
 m^\pm=m\bigl(t,k+(q\pm r)w\pm g\bigr),
\]
and set
\begin{equation}
\label{eq:deficit}
 D_o(t,q,r)=2m_o^0-m_o^+-m_o^-.
\end{equation}
Thus the increment is $k+qw$ and the background is $g+rw$.

\begin{lemma}[Positivity preservation]
\label{lem:preservation}
Suppose that, for every $o\in V$,
\begin{equation}
\label{eq:parabolic-boundary}
 D_o(0,q,r)\geq0\quad(q\geq0,\ r\in\R),
 \qquad
 D_o(t,0,r)\geq0\quad(0\leq t\leq1,\ r\in\R).
\end{equation}
Then $D_o(t,q,r)\geq0$ for all $o\in V$, $0\leq t\leq1$, $q\geq0$, and
$r\in\R$.
\end{lemma}

\begin{proof}
Bounded support justifies all differentiations, uniformly on compact
parameter sets. From \eqref{eq:interpolation},
\[
 \partial_t Z=\frac12\partial_w^2 Z.
\]
Hence,
\begin{equation}
\label{eq:magnetization-flow}
 \partial_t m_o(t,f)
 =\frac12\partial_w^2m_o(t,f)
  +\bigl(w\cdot m(t,f)\bigr)\partial_wm_o(t,f).
\end{equation}
Define
\[
 b_q=\frac12w\cdot(m^++m^-),\qquad
 b_r=\frac12w\cdot(m^+-m^-),\qquad
 c_o=\partial_qm_o^0.
\]
Substituting \eqref{eq:magnetization-flow} into \eqref{eq:deficit} gives
\begin{equation}
\label{eq:deficit-flow}
 \partial_tD_o
 =\frac12\partial_q^2D_o+b_q\partial_qD_o+b_r\partial_rD_o
  +c_o\sum_vw_vD_v.
\end{equation}
For the first-order terms, use
\[
 \begin{aligned}
 \partial_qD_o&=2\partial_qm_o^0-\partial_qm_o^+-\partial_qm_o^-,\\
 \partial_rD_o&=-\partial_qm_o^++\partial_qm_o^-,\\
 \sum_vw_vD_v&=2w\cdot m^0-w\cdot m^+-w\cdot m^-.
 \end{aligned}
\]
By \eqref{eq:fkg},
\begin{equation}
\label{eq:cooperative}
 c_o=\sum_vw_v\Cov_{t,k+qw}(x_o,x_v)\geq0.
\end{equation}
Thus the system \eqref{eq:deficit-flow} is cooperative: every coefficient
$c_ow_v$ coupling one component to another is nonnegative.
Furthermore, bounded support, Cauchy--Schwarz, and
$\Var(x_v)\leq M^2$ give the global bounds
\begin{equation}
\label{eq:bounded-coefficients}
 \abs{D_o}\leq4M,\qquad
 \abs{b_q},\abs{b_r}\leq M\sum_vw_v,\qquad
 0\leq c_o\leq M^2\sum_vw_v.
\end{equation}
To apply the maximum principle on the unbounded half-plane, set
\[
 \rho(q,r)=1+q^2+r^2,
 \qquad
 \mathcal L=\frac12\partial_q^2+b_q\partial_q+b_r\partial_r.
\]
Set $S=\sum_vw_v$. Since $\abs{q}+\abs{r}\leq\rho$,
\eqref{eq:bounded-coefficients} gives
\[
 \mathcal L\rho=1+2b_qq+2b_rr\leq(1+2MS)\rho,
 \qquad c_oS\leq M^2S^2.
\]
We may therefore take $A=2+2MS+M^2S^2$, so that, for every $o$,
\[
 A\rho-\mathcal L\rho-c_oS\rho\geq\rho>0.
\]
For $\delta>0$, put
$W_o=D_o+\delta e^{At}\rho$. Equation~\eqref{eq:deficit-flow} implies
\begin{equation}
\label{eq:strict-barrier}
 (\partial_t-\mathcal L)W_o-c_o\sum_vw_vW_v>0.
\end{equation}
By \eqref{eq:parabolic-boundary}, $W_o>0$ initially and on $q=0$.
The bound on $D_o$ also makes $W_o>0$ outside a fixed compact set, uniformly
in $t\in[0,1]$.

If some component became nonpositive, compactness and strict positivity at
$t=0$ would give a first zero at a time $t_*>0$ and a spatially interior
point with $q_*>0$. All components would be nonnegative there.
For the touching component, its time derivative would
be nonpositive, both first spatial derivatives would vanish, and its second
$q$-derivative would be nonnegative. In view of \eqref{eq:cooperative}, the
left-hand side of \eqref{eq:strict-barrier} would be nonpositive, a
contradiction. Letting $\delta\downarrow0$ proves the claim. No diffusion in
the $r$-direction is required, since the $r$-transport term vanishes at the
first contact point.
\end{proof}

\subsection{Completion of the proof of Theorem~\ref{thm:dss}}

For an increment $h\geq0$, set $V_+=\{v\in V:h_v>0\}$.
Following \cite[Section~2]{DSS23}, we prove Theorem~\ref{thm:dss} by
lexicographic induction on
\[
 (n,p)=(\abs{V},\abs{V_+}).
\]
For $p=0$, \eqref{eq:midpoint} is an equality by spin-flip symmetry. For
$n=1$, it is Lemma~\ref{lem:one-site}.

Fix $n\geq2$ and $1\leq p\leq n$. Assume the theorem for all models with
$n-1$ vertices and arbitrary nonnegative increments, and for all models with
$n$ vertices and fewer than $p$ positive coordinates. These hypotheses range
over all the allowed single-site measures and couplings.

Fix the target model, $g$, and $h$ with $\abs{V_+}=p$. Choose $v\in V_+$ and
set
\[
 k=h-h_ve_v.
\]
Then $k\geq0$ has $p-1$ positive coordinates. Fix $\varepsilon>0$ and define
\[
 w_v=\varepsilon^{-1/2},\qquad
 w_u=\varepsilon^{1/2}J_{vu}\quad(u\ne v),
\]
and let $K$ be obtained from $J$ by deleting all edges incident to $v$:
\[
 K_{uv}=0\quad(u\ne v),\qquad
 K_{u\ell}=J_{u\ell}\quad(u,\ell\ne v).
\]
Use this $K$ and $w$ in \eqref{eq:interpolation}. At time $t$, the single-site
laws are proportional to $e^{-(1-t)w_u^2x^2/2}\nu_u(\dd{x})$, so
Lemma~\ref{lem:one-site} keeps them in the permitted class. The couplings are
\[
 J_{vu}(t)=tJ_{vu},\qquad
 J_{u\ell}(t)=J_{u\ell}+t\varepsilon J_{vu}J_{v\ell}
 \quad(u,\ell\ne v,\ u\ne\ell).
\]
The additional edges between neighbors of $v$ are permitted because the
induction hypotheses range over all ferromagnetic couplings.

\paragraph{Initial time.}
At $t=0$, vertex $v$ is independent of the remaining $n-1$ vertices.
In \eqref{eq:deficit}, the increment is $k+qw\geq0$, and the background is
$g+rw$. If the observed site is $v$, apply the one-site comparison; otherwise,
apply the induction hypothesis on $n-1$ vertices. This gives
\[
 D_o(0,q,r)\geq0\qquad(o\in V,\ q\geq0,\ r\in\R).
\]

\paragraph{Boundary of the half-plane.}
At $q=0$, the increment in \eqref{eq:deficit} is exactly $k$. Every
intermediate model has $n$ vertices and belongs to the allowed class, while
$k$ has only $p-1$ positive coordinates. The induction hypothesis on $p$
therefore gives
\[
 D_o(t,0,r)\geq0\qquad(o\in V,\ 0\leq t\leq1,\ r\in\R).
\]
No induction hypothesis is used in the interior,
where $k+qw$ may have more than $p$ positive coordinates.

Both assumptions of Lemma~\ref{lem:preservation} now hold. Consequently,
\begin{equation}
\label{eq:interpolated-positive}
 D_o(t,q,r)\geq0
 \qquad(o\in V,\ 0\leq t\leq1,\ q\geq0,\ r\in\R).
\end{equation}

\paragraph{Recovery of the target model.}
Evaluate \eqref{eq:interpolated-positive} at
$(t,q,r)=(1,\sqrt\varepsilon\,h_v,0)$. The single-site penalties vanish.
The couplings incident to $v$ are exactly the target couplings; all other
couplings differ from the target by $\varepsilon J_{vu}J_{v\ell}$.
Moreover,
\[
 k+\sqrt\varepsilon\,h_vw
 =h+\varepsilon h_v\sum_{u\ne v}J_{vu}e_u.
\]
Consequently, as $\varepsilon\downarrow0$,
\[
 K_{u\ell}+w_uw_\ell\longrightarrow J_{u\ell}\quad(u\ne\ell),
 \qquad k+\sqrt\varepsilon\,h_vw\longrightarrow h.
\]
The Boltzmann factors at these parameters converge uniformly on the
bounded spin domain. Their integrals and first moments therefore converge,
and hence
\[
 0\leq D_o(1,\sqrt\varepsilon\,h_v,0)
 \longrightarrow 2m_o(h)-m_o(h+g)-m_o(h-g).
\]
Positivity holds for each $\varepsilon>0$, so no uniform bound on the
constants in the maximum principle is needed. This completes the induction
and the proof of Theorem~\ref{thm:dss}.\hfill$\square$


\clearpage
\begin{thebibliography}{9}

\bibitem{ABBD26}
Omar Abdelghani, Roland Bauerschmidt, Thierry Bodineau, and Benoit Dagallier.
\newblock The Ding--Song--Sun inequality for a class of even ferromagnets.
\newblock Preprint (2026).
\newblock \href{https://arxiv.org/abs/2609.08980}{arXiv:2609.08980}.

\bibitem{DSS23}
Jian Ding, Jian Song, and Rongfeng Sun.
\newblock A new correlation inequality for Ising models with external fields.
\newblock \emph{Probability Theory and Related Fields} \textbf{186} (2023),
477--492.
\newblock \href{https://doi.org/10.1007/s00440-022-01132-1}{doi:10.1007/s00440-022-01132-1}.

\bibitem{EMN76}
Richard S. Ellis, James L. Monroe, and Charles M. Newman.
\newblock The GHS and other correlation inequalities for a class of even ferromagnets.
\newblock \emph{Communications in Mathematical Physics} \textbf{46} (1976),
167--182.
\newblock \href{https://doi.org/10.1007/BF01608495}{doi:10.1007/BF01608495}.

\bibitem{FKG71}
C. M. Fortuin, P. W. Kasteleyn, and J. Ginibre.
\newblock Correlation inequalities on some partially ordered sets.
\newblock \emph{Communications in Mathematical Physics} \textbf{22} (1971),
89--103.
\newblock \href{https://doi.org/10.1007/BF01651330}{doi:10.1007/BF01651330}.

\end{thebibliography}
\end{document}